\documentclass[12pt, reqno]{amsart}
\usepackage{amssymb,amsmath,amsthm}

\usepackage[ frenchstyle, sfmathbb, partialup, nofligatures, ]{kpfonts}

\newtheorem{thm}[subsection]{Theorem}
\newtheorem{defn}[subsection]{Definition}
\newtheorem{lemma}[subsection]{Lemma}

\theoremstyle{definition}

\newtheorem{rmk}[subsection]{Remark}

\begin{document}

\title[Embeddings of Quadratic Spaces]{Embeddings of Quadratic Spaces}

\author{Vineeth Chintala}

\address{Indian Institute of Technology, Bombay, India}
\email{vineethreddy90@gmail.com}

\vspace*{-1cm}
\maketitle

\section {{{Introduction}}}

In this paper, we examine the concept of an \emph{embedding} of a quadratic space and analyze its connection to Spin groups.

The special case of the Hyperbolic space has been studied in ~\cite{CV} using what are called Suslin Matrices. This paper generalizes the results of ~\cite{CV} to any quadratic space, 
while at the same time giving a simpler ``basis-free" route to the main theorem (using Lemmas ~\ref{lemma1} - ~\ref{main}). These lemmas, as we shall see, help characterize a norm function and capture the Spin representation (Theorem ~\ref{main2}).
\vskip 1mm
By a quadratic space $(V,q)$,
we mean a free module $V$, over a commutative ring, equipped with a quadratic form $q.$ Without further ado,

\vskip 1mm
\begin{defn}
 Let $(V,q)$ be a quadratic space and $A$ be an associative algebra. We will say that $(V,q)$ is embedded in $A$ if $V \subseteq A$ and $$q(v) = v\alpha(v) = \alpha(v)v$$ 
where $\alpha$ is an isometry of $(V,q)$.
\end{defn}

Familiar examples of embeddings are given by the Composition algebras, Clifford Algebras and Suslin matrices. 
Given an embedding of a quadratic space, we connect it to the Clifford Algebra and derive some fundamental properties that any embedding should satisfy. Conversely, we can describe the Clifford Algebra and the corresponding Spin groups in terms of the algebra in which the quadratic space is embedded. It turns out that when there is an involution of $A$ that acts trivially on the underlying quadratic space $(V,q)$, the Spin group acts faithfully. 

The connection between the Clifford Algebra $Cl(V,q)$ and a general embedding $V \subseteq A$ is made by interpreting $Cl$ as a subalgebra of $M_2(A)$. 
This is achieved with the help of a structure theorem, about the graded ideals in Clifford algebras (Theorem ~\ref{injectivity}), that holds for any non-degenerate quadratic space.

Though Clifford Algebras have been studied in detail, they may not always be easy to work with. Sometimes it might be useful to switch to a more concrete embedding (as in the case of Suslin Matrices) to study low dimensional Spin and Epin (or Elementary Spin) groups. For instance, one can easily compute using Suslin Matrices, the exceptional isomorphisms $Spin(H(R^3)) \cong Sl_4(R)$ (Theorem 7.1, ~\cite{CV}) and $Epin(H(R^3)) \cong E_4(R)$ (Theorem 8.4, ~\cite{CV}). We conclude this paper with a brief description of Suslin matrices and use them to give an explicit set of generators of the Clifford algebras of some quadratic spaces.

\vskip 4mm

\subsection{Notation} 
Let $R$ be a commutative ring and $V$ be a free $R$-module with basis $\{v_1,\cdots v_n\}$. In addition, $V$ is equipped with a quadratic form $q$.
Then there is a bilinear form associated to $(V,q)$, given by $\langle v,w\rangle = q(v+w) - q(v) - q(w)$, for $v,w \in V$. 

\vskip 1mm
We say that $(V,q)$ is \emph{non-singular} 
if the matrix corresponding to the bilinear form, $B = \bigl(\langle v_i,v_j\rangle \bigr)$ is invertible.

\vskip 1mm
 A quadratic space $(V,q)$ is said to be \emph{non-degenerate} when $\langle x,v\rangle =0 $ holds for all $v \in V$, 
if and only if $x=0$.

\vskip 1mm

Suppose $(V,q)$ is degenerate. Then there exists a $x\in V$, $x\neq 0$ such that $\langle x,v_i\rangle =0 $ for $i \in \{1,\cdots,n\}$. Writing $x =(x_1,\cdots,x_n)$, 
this is equivalent to saying that there is a non-trivial solution to the matrix equation $xB =0$, where $B = \bigl(\langle v_i,v_j\rangle \bigr)$.
When $R$ is a commutative ring, the matrix equation 
$xB = 0$ has a non-trivial solution if and only if the determinant of $B = \bigl(\langle v_i,v_j\rangle \bigr)$ is a zero divisor 
(~\cite{MD}, Corollary I.30.). 
In other words, $(V,q)$ is non-degenerate if and only if $det\bigl(\langle v_i,v_j\rangle \bigr)$ is a non-zero divisor.
\vskip 1mm

Unless otherwise stated, all quadratic spaces in the paper are assumed to be non-degenerate. 
All modules considered are free-modules over a commutative ring $R$.

 \subsection{General References}
 For general literature on Clifford algebras and Spin groups over a commutative ring, the reader is referred to
 ~\cite{B2}, ~\cite{K}. To learn more about Suslin Matrices, see ~\cite{RJ} and Chapter III.7, ~\cite{L}.

\section{{Preliminaries on Clifford Algebras}}

Perhaps the most important example of an embedding of a quadratic space is given by the Clifford algebra. Given any quadratic space $(V,q)$, its Clifford algebra $Cl(V,q)$ (or simply, $Cl$) is the ``freest" algebra generated by $V $ subject to the condition 
$x^2 = q(x)$ for all $x \in V$. 

More precisely, the Clifford algebra $Cl(V,q)$ is the quotient of the tensor algebra $$T(V) = R \oplus V\oplus V^{\otimes 2} \oplus \cdots \oplus V^{\otimes n} \oplus \cdots$$ 
by the two sided ideal $I(V,q)$ generated by all $x\otimes x - q(x)$ with $x \in V$. 
\vskip 4mm

\subsection {Basic Properties of the Clifford Algebra $Cl(V,q)$ :}

\begin{itemize}
\vskip 3mm
\item \underline{$Z_2$ Grading of $Cl$} :  The Clifford algebra $Cl(V,q)$ is an associative algebra (with unity) over $R$  with a linear map 
$i : V \rightarrow Cl(V,q)$ such that $i(x)^2 = q(x)$. The terms $x\otimes x$ and $q(x)$ appearing in the generators of $I(V,q)$ have degrees $0$ and $2$ in the grading of $T(V)$. 
By grading $T(V)$ modulo $2$ by even and odd degrees, it follows that the Clifford algebra has a $Z_2$-grading $Cl =  Cl_0\oplus Cl_1$ such that $V \subseteq Cl_1$ and $Cl_iCl_j\subseteq Cl_{i+j}$ ($i,j$ mod $2$). 
\vskip 2mm

\item \underline{Universal Property}: The Clifford algebras has the following universal property. Given any associative algebra $A$ over $R$ and any linear map 
$j : V \rightarrow A$ such that 
$$ j(x)^2 = q(x) \text{ for all $x \in V$},$$
there is a unique algebra homomorphism $f : Cl(V,q) \rightarrow A$ such that 
$f \circ i = j$.
 \vskip 1mm

\item  \underline{Basis of $Cl$}: The elements of $V$ generate the Clifford algebra. Furthermore, the following result implies that if $rank(V) =n$, then $rank(Cl) = 2^n$. \begin{thm}(Poincar\'e-Birkhoff-Witt)
 \\Let $\{v_1,\cdots, v_n\}$ be a basis of $(V,q)$. Then $\{v_1^{e_1}\cdots v_n^{e_n} : e_i =0,1\}$ is a basis of $Cl(V,q)$. 
  \end{thm}
For a simple proof, see Theorem IV. 1.5.1, ~\cite{K}.
\end{itemize}

The above theorem implies in particular that the map $i : V \rightarrow Cl(V,q)$ is injective. Since $i(v)^2 =q(v)$ for $v \in V$, it is clear that $i(V) \subseteq Cl$ is an embedding. We will refer to this as the Clifford embedding of $(V,q)$.

\vskip 2mm

\vskip 3mm
\subsection{Structure of Clifford algebras}

To analyze general embeddings, we need one more result about Clifford algebras, and this will be stated in Theorem ~\ref{injectivity}. The theorem says that a graded homomorphism from $Cl$ is injective whenever its restriction on $R$ is injective. 
The reader may assume that and jump ahead to the next section.

\vskip 1mm
For any $Z_2$-graded-algebra $A = A_0+ A_1$, the elements in $h(A) = A_0 \cup A_1$ will be called the \emph{homogeneous} elements of $A$. If $a \in h(a)$, we write $\partial(A) = i$ if $a \in A_i$, ($i =0,1$).

\noindent The graded tensor product of two algebras $A \hat{\otimes} B$ is defined as :
\[  (a\hat{\otimes} b)(a'\hat{\otimes} b')= (-1)^{\partial(b) \partial(a')}aa'\hat{\otimes} bb' \]
for all homogeneuos elements $a,a' \in A$ and $b, b' \in B$.
\vskip 2mm
One can use the universal property of Clifford Algebras to compute the Clifford algebra of an orthogonal sum of quadratic spaces, (see ~\cite{K},  Theorem IV.1.3.1.).

\begin{thm}\label{thm0}
The map $f:  (V_1,q_1) \perp (V_2, q_2)  \rightarrow Cl(V_1,q_1) \hat{\otimes}  Cl(V_2,q_2)$ defined by $f(x_1 +x_2) = x_1\hat{\otimes}1 + 1 \hat{\otimes}x_2$ induces an isomorphism \[Cl(V_1\perp V_2 ,q_1 \perp q_2)  \cong Cl(V_1,q_1) \hat{\otimes}  Cl(V_2,q_2).\]
\end{thm}

The hyperbolic module $H(V) = V \oplus V^*$ is equipped with a quadratic form $q(x,f) =f(x)$. This quadratic space will be referred to as the hyperbolic space. We will return to them in the final section of the paper.

\begin{thm}\label{thm00}
Let $(V,q)$ be a non-singular quadratic space where $V$ is a free module with rank $n$. Then $(V,q) \perp (V, - q)  \cong H(V)$ and 
\[Cl(V,q) \hat{\otimes}  Cl(V,-q) \cong M_{2^n}(R)\]
\end{thm}

\begin{proof}
For a proof that $ (V,q) \perp (V, - q)$ and $H(V)$ are isomorphic as quadratic spaces, see  ~\cite{B1},  Ch. 5, Lemma 2.2.

It is also known that $Cl(H(V)) \cong M_{2^n}(R)$ (Theorem 7.1.10, ~\cite{HOM}). The result follows using Theorem ~\ref{thm0}.
\end{proof}

An ideal $J \subseteq Cl$ is said to be graded if it is a direct sum of the intersections $J_i = J\cap Cl_i$. 
A homomorphism between two $Z_2$- graded algebras is said to be graded if it preserves the grading. 
The kernel of a graded homomorphism is a graded ideal.

\begin{thm}\label{thm000}
Let $(V,q)$ be a non-degenerate quadratic space. Let $J$ be a graded ideal of $Cl(V,q)$ such that 
$J \cap R =\{0\}$. Then $J =\{0\}$.
\end{thm}
\begin{proof}

\textbf{Case 1 : $(V,q)$ is non-singular.}
\vskip 1mm
Consider the graded homomorphism $\phi : Cl(V,q) \rightarrow Cl(V,q)/J$. By Theorem ~\ref{thm00}, the map $\phi$ extends to a graded homomorphism 
$\phi' : M_{2^n}(R) \rightarrow Cl(V,q)/J \hskip 1mm   \hat{\otimes} \hskip 1mm Cl(V,-q) $.  
Now every ideal of $M_{2^n}(R)$ is of the form $M_{2^n}(I)$ for some ideal $I$ in $R$. Since $J \cap R =\{0\}$, this is possible only if $I=\{0\}$.
\vskip 1mm

\textbf{Case 2 : $(V,q)$ is non-degenerate.}
\vskip 1mm
When $(V,q)$ is non-degenerate, $d = det \bigl(\langle v_i,v_j\rangle \bigr)$ is a non-zero divisor in $R$ (see Section 1.2). 
Then $V \otimes R[d^{-1}]$ is a non-singular quadratic space and so $J \otimes R[d^{-1}]$ is the zero ideal by Case 1. 
Therefore $J = \{0\}$.
\end{proof}

\begin{thm}\label{injectivity}
Let $(V,q)$ be a non-degenerate quadratic space. 
Let $\phi : Cl(V,q) \rightarrow A$ be a graded (algebra) homomorphism such that $ker(\phi) \cap R =\{0\}$. 
Then $\phi$ is injective.
\end{thm}

\begin{proof}
The kernel of $\phi$ is a graded ideal in $Cl(V,q)$. The result follows immediately from Theorem ~\ref{thm000}. 
\end{proof}

Note that Theorem ~\ref{injectivity} allows the possibility of $d = det \bigl(\langle v_i,v_j\rangle \bigr)$ being a zero divisor in $A$, even though $d$ is a non-zero divisor in $R$.

\vskip 3mm

\section{{Basic properties of Embeddings}}
\vskip 5mm
 \subsection{Connecting two different embeddings :}

 Let $Cl$ denote the Clifford algebra of $(V,q)$. To avoid any confusion, we will denote the copy of $V$ in its Clifford algebra by $V_{Cl}$.  

\vskip 1mm
Let $ (V, q) \subseteq A$ be an embedding with $q(v) = v\overline{v} = \overline{v}v$, for some isometry $v \rightarrow \overline{v}$, for $v \in V$. 
\vskip 1mm

Let $\phi : V_{Cl} \rightarrow M_2(A)$ defined by
$ v_{cl} \rightarrow  \bigl(\begin{smallmatrix}
0                  & v \\
\overline{v}          & 0 \end{smallmatrix}\bigr).$
As $\phi^2(v_{cl}) =  q(v)$ for all $v \in V$, the map $\phi$ extends to an $R$-algebra homomorphism 
$\phi : Cl \rightarrow M_2(A)$.  This is in fact a graded homomorphism, where the even and odd elements of $M_2(A)$ correspond to matrices of the form
$(\begin{smallmatrix}
*    &  0  \\
0       &*  \end{smallmatrix})$ and $(\begin{smallmatrix}
 0   &  *  \\
*        & 0 \end{smallmatrix})$. 

Let $ker(\phi)$ denote the kernel of $\phi$.  Since $\phi$ restricts to an injective map on $V_{Cl}$, we have  $ ker(\phi) \cap R = \{0\}$. Therefore it follows from Theorem ~\ref{injectivity} that $\phi$ is injective. 
\vskip 2mm

From here on, we will identify $Cl$ with its image in $M_2(A)$.

\begin{thm}\label{thm1}
Let $(V,q)$ be a quadratic space embedded in an algebra $A$. Let $v, w \in V$. Then $vwv\in V$ and $\overline{vwv} = \overline{v}\cdot \overline{w}\cdot \overline{v}.$
\end{thm}

\begin{proof}
We will first prove the theorem for the Clifford embedding. Let $z_1, z_2 \in V_{Cl}$. Then $$\langle z_1, z_2\rangle := (z_1 + z_2)^2 - z_1^2 - z_2^2 = z_1z_2 + z_2z_1$$ is an element in $R$. 
 Multiplying by $z_1$ we get $$z_1\langle z_1,z_2 \rangle= z_1^2z_2 + z_1z_2z_1.$$
 
 \vskip2mm
 
Since $z_1^2 = q(z_1)$, it follows that $z_1z_2z_1 \in V_{Cl}$. For any embedding $V \subseteq A$, there is a map $\phi : Cl \rightarrow M_2(A)$ given by $v_{cl} \rightarrow  \bigl( \begin{smallmatrix}
0          & v\\
\overline{v}     & 0 \end{smallmatrix}\bigr) $. 

\vskip 2mm

Take 
$z_1 \rightarrow  \bigl( \begin{smallmatrix}
0          & v\\
\overline{v}     & 0 \end{smallmatrix}\bigr) $ 
and $z_2 \rightarrow  \bigl(\begin{smallmatrix}
0          & \overline{w}\\
w     & 0 \end{smallmatrix}\bigr) $.
Then $z_1z_2z_1 
\rightarrow \bigl(\begin{smallmatrix}
0                             &  vwv \\
\overline{v}\overline{w} \overline{v} & 0 \end{smallmatrix}\bigr) .$
\end{proof}

Given an embedding, one can also treat $(V,q)$ as a non-associative algebra, with its multiplication given by $v \bullet w = vwv$ for $v,w \in V$. Under this multiplication, $(V,q)$ 
becomes a Quadratic Jordan algebra. By the above theorem, $\overline{v \bullet w} = \overline{v} \bullet \overline{w}$. Then the bijection $v \rightarrow \overline{v}$ in $(V,q)$ is not only an isometry, but 
also an algebra automorphism of $(V,q)$. In addition, if $1 \in V$, then $(1+v)w(1+v) \in V$, and so $vw +wv \in V$. A general theory of Quadratic Jordan algebras can be found in ~\cite{J, M}. 

\subsection{The Standard involution}

The map $ v_{Cl}\rightarrow - v_{Cl}$ can be viewed as an inclusion of $V_{Cl}$ in the opposite algebra of $Cl$. By the universal property of the Clifford algebra, 
this map extends to an involution $*$ of $Cl$. This is called the standard involution on $Cl$. In terms of the algebra $M_2(A)$, we have
$(\begin{smallmatrix}
0      &  v  \\
\overline{v}        &  0 \end{smallmatrix})^* =
(\begin{smallmatrix}
0      &  -v  \\
-\overline{v}       &  0 \end{smallmatrix})$.
We will see below that under certain circumstances, the standard involution on $Cl$ restricts to an involution of $A$.

\vskip 1mm

One might also wonder if the map $v \rightarrow -v$ or the isometry $\alpha(v) = \overline{v}$ (for $v\in V$) extends to an involution of $A$. 
We will keep returning to similar questions in the paper. For now it is not even clear what values, the order of $\alpha$ can take. 

In the Clifford embedding, the isometry of $V$ is simply the identity map. For Composition algebras and (as we'll see later) Suslin matrices, the isometry corresponding to the embedding has order $2$. This is not an accident and we will now show that this is true whenever $1_A\in V$ and $\overline{1_A} = 1_A$.

\vskip 1mm
Suppose  $1_A\in V$ and $ \overline{1_A} = 1_A$. Then we have the following nice implications :

\begin{itemize}
\item Let $v \in V$. Since $q(1+v) = (1+v)(\overline{1+v})$ is a scalar, so is $v+ \overline{v}$.

\item
We have \[\alpha^2(x)\alpha(x) = \alpha(x)x, \hskip 3mm \alpha^2(x+1)\alpha(x+1) = \alpha(x+1)(x+1).\] Using the fact that $\alpha$ is linear and 
cancelling terms, it follows that $\alpha^2(x) = x$ for all $x \in V$, i.e. $\alpha$ has order $2$.

\item  Now suppose $A \subseteq Cl$, i.e., $Cl$ contains all elements $ (\begin{smallmatrix}
a      &  0  \\
0        &  a \end{smallmatrix})$ with $a \in A$; Then we will show that the standard involution restricts to $A$.
\vskip 1mm
First notice that the even part of the Clifford algebra $Cl_0$ is closed under the standard involution, and its image in $M_2(A)$  consists of matrices of the form 
$(\begin{smallmatrix}
x      &  0  \\
0        &  y \end{smallmatrix})$. 
We will simply write $(x,y)$ instead of $(\begin{smallmatrix}
x      &  0  \\
0        &  y \end{smallmatrix})$.

\vskip 1mm
Let $(a,a)^* = (x,y)$. For $A$ to be closed under the involution, we need $x =y$. 

Let $e = (\begin{smallmatrix}
0      &  1  \\
1        &  0 \end{smallmatrix})$. We have $e^* = -e$ as $e \in V_{M_2(A)}$.
Since $e(a,a)e = (a,a)$, we have $e^*(a,a)^*e^* = (a,a)^*$.
Therefore $$(y,x) =e \cdot (a,a)^* \cdot e = (x,y)$$
and so the standard involution on $Cl$ restricts to $A$.

\end{itemize}

Conversely, given an involution of $A$, one might ask if it can be extended to the standard involution on $Cl$.
To explore this possibility, let us write $M =   \bigl(\begin{smallmatrix}
a          & b\\
c          & d \end{smallmatrix}\bigr) $ as a $2 \times 2$ matrix and 
analyze its conjugate in terms of its blocks. The table below illustrates, for a few examples, how the action of an involution $*$ on $A$ can be extended to the standard involution in $Cl$, which is seen as a sub-algebra of $M_2(A)$.

Note that in order
to show that an involution corresponds to the standard involution of the Clifford algebra, it is enough to check that its action on the elements of $V_{Cl}$ is multiplication by $-1$. In other words, if $ z = \bigl(\begin{smallmatrix}
0      &  v  \\
\overline{v}        &  0 \end{smallmatrix}\bigr)$, then we need $z^* = -z$. This is clearly the case for both the involutions defined in the table below.

\vskip 3mm

\begin{center}\label{table1}

\begin{tabular}{ |c|c|c|c| }
\hline
\multicolumn{4}{|c|}{Standard Involution on $M =   \bigl(\begin{smallmatrix}
a          & b\\
c          & d \end{smallmatrix}\bigr) $} \\
\hline
&&&\\
Form 1: & $v^* = u\cdot v$& $u^2 =1$ , $u \in R$ &  $M^* = \bigl(\begin{smallmatrix}
d^*          & -ub^*\\
-uc^*          & a^* \end{smallmatrix}\bigr) $\\
&&&\\
Form 2 : &$v^* = u\cdot \overline{v}$& $u^2 =1$ , $u \in R$ &  $ M^* =\bigl(\begin{smallmatrix}
a^*          & -uc^*\\
-ub^*          & d^* \end{smallmatrix}\bigr) $\\
 &&&\\
   \hline
\end{tabular}
\end{center}
\vskip 2mm 
Since the involution acts trivially on scalar matrices, notice that if $1_A \in V$, and $v^* = uv$ or $v^* = u \overline{v}$, then it follows that $u =1$.

\vskip 2mm


\section{{The Spin Representation: When $v^*= v$ for all $v \in V$}}

Motivated by the discussion in the previous section, we will now analyze embeddings $(V,q) \subseteq A$ with the following conditions:
\begin{enumerate}  
\item  $1_A \in V_A$ and $\overline{1_A} = 1_A $. 
\vskip 2mm
\item There is an involution $*$ of $A$ that restricts to the identity map on $V$, i.e, $v^*= v$ for all $v \in V$. 
 \end{enumerate}    
 
We will continue to identify $Cl$ as a sub-algebra of $M_2(A)$.  Let $M =  \bigl(\begin{smallmatrix}
a          & b\\
c          & d \end{smallmatrix}\bigr)\in M_2(A)$.
Then $M^* = 
  \bigl(\begin{smallmatrix}
d^*          & -b^*\\
-c^*          & a^* \end{smallmatrix}\bigr) $ gives us the standard involution on $Cl$. 
\vskip 1mm
In particular, $\big(\begin{smallmatrix}
g_1      &  0 \\
0        &  g_2 \end{smallmatrix}\big)^* = \big(\begin{smallmatrix}
g_2^*      &  0 \\
0        &  g_1^* \end{smallmatrix}\big).$ For convenience, we will sometimes write $(g_1,g_2)$ instead of $(\begin{smallmatrix}
g_1      &  0 \\
0        &  g_2 \end{smallmatrix}).$
     
\vskip 2mm

\subsection{Spin group}
The following groups are relevant to our discussion : 
    \[U^0(V) := \{x \in Cl_0 \, | \,  xx^* = 1\}\]
    \[ Spin(V):= \{x \in U^0(R) \,| \, xV_{Cl}x^{-1} = V\}.\] 
\vskip 1mm
Notice that the action of the Spin group on $V$ is an isometry of $V$.

\subsection{The Spin representation :}

We will define a group
$SG(A) \subset A$ and show that it is isomorphic to the Spin group, when $v^*= v$ for all $v \in V$. 
\vskip1mm
Let $(g_1, g_2) \in Spin (V)$. Since $(g_1, g_2)^* = (g_2^*, g_1^*)$, we have 
$g_2 = g_1^{*-1}.$
\vskip 1mm

Let $(g, g^{*{-1}}) \in Spin(V)$ and $v \in V$. By definition, 
 there exists an element $w \in V$ such that 
$(g, g^{*{-1}}) \big(\begin{smallmatrix}
0                  & v \\
\overline{v}           & 0 \end{smallmatrix}\big)
(g^{-1}, g^*) = 
 \big(\begin{smallmatrix}
0                  & w \\
\overline{w}          & 0 \end{smallmatrix}\big) $, i.e.,

\[  \begin{pmatrix}
0                  & gvg^* \\
g^{*{-1}}\overline{v}g^{-1}           & 0 \end{pmatrix}  = 
 \begin{pmatrix}
0                  & w \\
\overline{w}         & 0 \end{pmatrix} .\]
\vskip 3mm

\noindent {Let $g\bullet v = gvg^*$ for $g\in A$.}  \vskip 3mm

\noindent If $(g, g^{*{-1}}) \in Spin(V)$, then $g\bullet v \in V$.
 Let $A^{\times}$ denote the invertible elements of $A$. Consider the set \[G(A) = \{\ g\in A^{\times} \ | \text{ $g\bullet v \in V$ 
 $\forall$ $v \in V$}\}.\]
 
Since the action $\bullet$ is bijective, it is easy to see that $G(A)$ is a group and is closed under the involution $^*$.

\vskip 1mm
\noindent One has the homomorphism
$$\chi : Spin(V) \rightarrow G(A)$$ given by $(g,{g^*}^{-1}) \rightarrow g.$
\vskip 1mm

Next, we will use the quadratic form $q$ on $V$ to define a `norm' on $G(A)$. The Spin group will turn out to be isomorphic to the subgroup of $G(A)$ whose elements have unit norm.
\vskip 2mm
Let us begin with three simple lemmas which help us show that \[q(g\bullet v) = q(gg^*)q(v), \text{\hskip 3mm for $g \in G(A).$}\]

\vskip 2mm
\begin{lemma}\label{lemma1}
Let $v \in V$ such that \{$v$, $1$\} are linearly independent. Suppose there exists an element $v' \in V$ such that $v +v'$ and $vv'$ are scalars. Then $v' = \overline{v}$.
\end{lemma}

\begin{proof}
Since $v +\overline{v} \in R$, it follows that $\overline{v} = v' +r$ for some $r \in R$. Since $v\overline{v} = q(v)$, it follows that $v(v' + r) \in R$, implying $rv \in R$. Therefore $r=0$ and $v' = \overline{v}$.
\end{proof}

\vskip 3mm
\begin{lemma}\label{lemma2}
Suppose $v_1, v_2 \in V$ and $q(v_2) = v_2\overline{v_2}=1$. Then $\overline{v_1} +  v_2v_1v_2 \in v_2R$.
\end{lemma}

\begin{proof}
Since  $(\overline{v_1} +v_2)(v_1 + \overline{v_2}) \in R$, it follows that
\[ \langle \overline{v_1}, {v_2}\rangle =  \overline{v_1}\cdot \overline{v_2} + v_2v_1  \in R.\] Multiplying by $v_2$ on the right, this implies that 
\[\overline{v_1} +  v_2v_1v_2 \in v_2R. \qedhere \]
\end{proof}

\vskip 3mm
\begin{lemma}\label{lemma3}
Let $g\in G(A)$. If $q(gg^*) =1$, then $q(g^*g)  =1.$
\end{lemma}

\begin{proof}
Let $X = g^*g$. 
\vskip 1mm 
If $X \in R$, then $gg^* = g^*g $ and we are done. 
Suppose $X \notin R$. 
We will show that $X + X^{-1} \in R$ and infer from Lemma $~\ref{lemma1}$ that $X^{-1} = \overline{X}$. 
Now, \[X^{-1} = g^{-1}g^{*{-1}} = g^* \bullet (g^{*{-1}} g^{-1})^2.\]
Since $q(gg^*) = 1$, we have $q(g^{*{-1}} g^{-1}) = q(\overline{gg^*})=1.$
Therefore (using Lemma $~\ref{lemma2}$),
\begin{equation*}
\begin{split}
X + X^{-1} & = g^* \bullet (1 + (g^{*{-1}} g^{-1})^2 ) \\
 & = g^* \bullet (rg^{*{-1}} g^{-1})  \text{,    \hskip 5mm         for some $r\in R$.} \\
 &= r \in R  
  \qedhere
 \end{split}
\end{equation*}
\end{proof}
\vskip 5mm

\begin{lemma} \label{main} 
Let $g \in G(A)$. For all $v \in V$, we have $$q(g\bullet v) = q(gg^*)q(v).$$
\end{lemma}
\vskip 2mm

\begin{proof}
\textbf{Case 1 : $q(gg^*) = 1.$}
\vskip 2mm
Let $w = g\bullet v$ and  $w' = g^{*{-1}}\bullet \overline{v}$. Since $w\cdot w' = v\overline{v}$, it is enough to prove that $\overline{w} = w'$.
\vskip 1mm
Let us assume for now that  $\{w,1\}$ are linearly independent. We will first show that $w +w' \in R$ and use it to prove that $w' = \overline{w}$.  
\vskip 1mm
Let $X=g^*g$. Since $q(gg^*) = 1$, it follows from Lemma ~\ref{lemma3} that $q(X)  =1$. We have
\begin{equation*}
\begin{split}
w &= gvg^* \\
&=  g^{*{-1}}(XvX)g^{-1} \\
&= g^{*{-1}} \bullet (XvX).\\
 \end{split}
\end{equation*}
\vskip 1mm

Since $q(X) = 1$, we know (from Lemma $~\ref{lemma2}$) that
$\overline{v} + XvX = rX$ for some $r \in R$. Therefore
\begin{equation*}
\begin{split}
w' + w & = g^{*{-1}} \bullet (\overline{v} + XvX) \\
 & = g^{*{-1}} \bullet rg^*g \\
 &= r \in R
 \end{split}
\end{equation*}

Since $w' + w$ and $ww'$ are scalars, it follows from Lemma ~\ref{lemma1} that $$w' = \overline{w}.$$  
Now suppose $\{w,1\}$ are linearly dependent. Then we can write $w = (w_0+w) - w_0$, where  $\{w_0,1\}$ are linearly independent. As $\bullet$ is a linear action it follows that $w' = \overline{w}$.

\vskip 2mm
\textbf{Case 2 : $q(gg^*) = a.$}
\vskip 1mm
Clearly $a$ is invertible since $g \in G(A)$. 

Suppose there is an $x \in R$ such that $x^2 = a^{-1}$. Take $h = xg$. Then $ q(h\bullet v) = x^2\cdot q(g\bullet v) $ and  $q(hh^*) =1$. The result follows immediately from Case 1.
\vskip 2mm
Now suppose $x^2 = a^{-1}$ has no solutions in $R$. Then one has the identity $q(g\bullet v) = q(gg^*)q(v)$ over the ring $\frac{R[x]}{(x^2 - a^{-1})}$; Since each term of the equation 
lies in $R$, the result follows in this case too.
\end{proof}

\vskip 2mm

We can now describe the Spin group as a group inside $A$. Let $R^\times$ denote the group of invertible elements in $R$.

\begin{thm}\label{main2}
Define $d : G(A) \rightarrow R^\times$ as $d(g) = q(gg^*).$ 
\vskip 1mm
Then $d$ is a group homomorphism and \[ ker (d) = SG(A) \cong Spin(V). \] 
\end{thm}

\begin{proof}
As a consequence of Lemma ~\ref{main} we have, for $g,h \in G(A)$,
\[d(gh) = q(ghh^*g^*) = q(gg^*)q(hh^*) = d(g)d(h).\] Thus $d$ is a group homomorphism and
\[
 ker (d) = SG(A) \cong Spin(V)
 \]
 in the case when the identity map on $V$ can be lifted to an involution of $A$.
\end{proof}
\vskip 4mm


\section{{The Suslin embedding}}
\vskip  2mm

Let $R$ be any commutative ring and $H(R^n) := R^n \oplus R^{n*}$. By fixing a basis of $R^n$, 
one can then write the quadratic form on $H(R^n)$ as 
$$q(v,w) = v\cdot w^T = a_1b_1 + \cdots + a_nb_n.$$
for $v =(a_1,\cdots,a_n)$, $w =(b_1,\cdots,b_n)$. This quadratic space $(H(R^n), q)$ is referred to as the hyperbolic space. We will now define Suslin matrices which give an embedding of the hyperbolic space into the ring of matrices $M_{2^{n-1}}(R)$.

\vskip 1mm

The Suslin matrix $S_n(v,w)$ of size 
$2^n \times 2^n$
is constructed 
from two vectors $v,w$ in $R^{n+1}$ as follows : 
\vskip 2mm
Let $v = (a_0, v_1)$, $w = (b_0, w_1)$ where $v_1$, $w_1$ are vectors in $R^n$. Define
{
\[  S_1(v,w) =  \begin{pmatrix}
a_0     &v_1    \\
-w_1    & b_0 \end{pmatrix} \hskip 5mm   \overline{S_1(v,w)} =  \begin{pmatrix}
b_0     & -v_1    \\
w_1    & a_0 \end{pmatrix}   \]

 \[ 
 S_n(v,w) =  \begin{pmatrix}
a_0I_{2^{n-1}}      &  S_{n-1}(v_1,w_1) \\
-\overline{S_{n-1}(v_1,w_1)}      &  b_0I_{2^{n-1}}\end{pmatrix}\]
and 
\[\overline{S_n(v,w)} =  \begin{pmatrix}
b_0I_{2^{n-1}}      &  - S_{n-1}(v_1,w_1) \\
\overline{S_{n-1}(v_1,w_1)}      &  a_0I_{2^{n-1}}\end{pmatrix}\]

It easily follows that $S_n = S_n(v,w)$ satisfies the following properties: 

{\small
\begin{enumerate}
\item $S_n\overline{S_n} = \overline{S_n}S_n = (v\cdot w^T) I_{2^n}$,
\item $\det S_n = (v \cdot w^T)^{2^{n-1}}$, for $n \geq 1$.
\end{enumerate}
}

\noindent In his paper ~\cite{S}, A. Suslin then describes a sequence of matrices $J_n \in M_{2^n}(R)$ such that $JJ^T = I$,
\begin{equation}\label{eq1}
 J_{n-1}S_{n-1}^TJ_{n-1}^T = 
\begin{cases}
S_{n-1} &\text{ for $n$ odd,}\\\\
\overline{S_{n-1}} &\text{ for $n$ even.}
 \end{cases}
\end{equation}
Clearly $M^* =JM^TJ^T$ is an involution of $M_{2^n}(R)$ (as $JJ^T =1$).Thus there are two types of involution for the Suslin embedding, depending on the parity of $rank(V) = n$.
\vskip 1mm
The map
$\phi : H(R^n) \rightarrow M_{2^n}(R)$ defined by 
$\phi(v,w) =  \bigl(\begin{smallmatrix}
0                  & S_{n-1}(v,w) \\
\overline{S_{n-1}(v,w)}         & 0 \end{smallmatrix}\bigr)$ 
induces an $R$-algebra homomorphism $\phi : Cl \rightarrow M_{2^n}(R)$.
In fact $\phi$ is an isomorphism (Section 3.1, ~\cite{CV}); the elements $\phi(v,w)$ give a set of generators of the Clifford algebra.

\begin{rmk}\label{remark2}
There cannot be two involutions $*_1, *_2$ of  $M_{2^n}(R)$ (for a fixed $n$) such that $S^{*_1} = S$ and $S^{*_2} = \overline{S}$. Otherwise both involutions can be lifted to the standard involution as in Table ~\ref{table1}.
This is not possible as the two (lifted) involutions act differently on 
$\big(\begin{smallmatrix}
S(v,w)                & 0\\
0        & S(v,w) \end{smallmatrix}\bigr) \in Cl $, when $\overline{S(v,w)} \neq S(v,w)$.
\end{rmk}

 \begin{rmk}
When $v\cdot w^T =1$, the kernel of the map $R^n \rightarrow R$, defined by $w \rightarrow  v\cdot w^T$, is a projective module. This projective module is not isomorphic to its dual when the row $v$ has odd size $>3$. However, these projective modules are self-dual when $n$ is even. (See  ~\cite{NRS}).
 
 Perhaps this difference in \emph{duality} in the odd and even cases can give a deeper explanation for the corresponding behavior of the Suslin Matrices, described above in Equation ~\ref{eq1}. 
 \end{rmk}

To learn more about the connection between Suslin Matrices and Clifford Algebras, see ~\cite{CV}. Suslin matrices were first introduced by A. Suslin in his paper ~\cite{S},  in connection with unimodular rows and K-Theory.  The recent work of A. Asok and J. Fasel (see ~\cite{AF}) uses Suslin Matrices in the context of $\mathbb{A}^1$-homotopy theory and Bott periodicity. 

\subsection{Applications to Quadratic Spaces} In a similar fashion, using Suslin Matrices, one can construct an explicit set of generators of the Clifford Algebra for other classes of quadratic spaces.  Here are a few examples :
 \vskip 2mm
 \noindent
 {\small
\begin{tabular}{ |p{4.5cm}|p{5.7cm}|p{2.2cm}| }

 \hline
  Quadratic Space $(V,q)$ & Clifford Embedding &$Cl$ \\
 \hline
 &&\\
   $(R^{2n}, \sum\limits_{i=1}^{n} v_iw_i$)&   $ (v,w) \rightarrow \Bigl(\begin{smallmatrix}
   0    & S(v,w)\\\\
\overline{S(v,w)}         & 0\end{smallmatrix}\Bigr)$ & $M_{2^n}(R)$\\
 &&\\
   $(R^{2n+1}, -v_0^2 + \sum\limits_{i=1}^{n} v_iw_i$)&   $ (x,v,w) \rightarrow \Bigl(\begin{smallmatrix}
v_0\lambda_1         & S(v,w)\\\\
\overline{S(v,w)}         & -v_0\lambda_1 \end{smallmatrix}\Bigr)$ & $M_{2^n}(R[\lambda_1])$\\
&&\\

 $(R^{2n+2}, -v_0^2 - w_0^2 + \sum\limits_{i=1}^{n} v_iw_i)$&   $(x,y,v,w) \rightarrow \Bigl(\begin{smallmatrix}
v_0\lambda_1 +w_0\lambda_2         & S(v,w)\\\\
\overline{S(v,w)}         & -v_0\lambda_1 - w_0\lambda_2 \end{smallmatrix}\Bigr)$ & $M_{2^n}(R[\lambda_1,\lambda_2])$\\\
&& \\
\hline
\end{tabular}
\vskip 1mm
(In the above table, we have $\lambda_1^2 =\lambda_2^2 =-1$ and $\lambda_1\lambda_2+\lambda_2\lambda_1=0$.)
}

\section{{Conclusion}}

We have analyzed general embeddings of quadratic spaces by connecting them to their respective Clifford algebras. 
Given an embedding $(V,q) \subseteq A$, we know that $vwv \in V$ and so $V$ becomes a Quadratic Jordan algebra. It would be very interesting to learn the conditions under which the identity map on $V \subseteq A$ (or more generally, maps of the type $v^*= uv$, $u\in R$)  can be lifted to an involution of $A$.  We know that this is not always the case for Suslin matrices (see Remark ~\ref{remark2}). 

The classification and structure of Special Jordan algebras has been worked out in the 20th century (see ~\cite{M} for a survey). But their relationship with the overlying associative algebras, in different embeddings, remains to be explored further.

\end{document}